\documentclass{amsart}
\usepackage{amssymb}
\usepackage[mathscr]{eucal}

\newtheorem{theorem}{Theorem}[section]
\newtheorem{lemma}[theorem]{Lemma}

\newtheorem*{proposition*}{Proposition}
\newtheorem{corollary}[theorem]{Corollary}
\theoremstyle{definition}

\newtheorem{example}[theorem]{Example}

\theoremstyle{remark}
\newtheorem{remark}[theorem]{Remark}

\numberwithin{equation}{section}

\newcommand{\bbR}{\mathbb{R}}
\newcommand{\bbC}{\mathbb{C}}

\newcommand{\bbZ}{\mathbb{Z}}

\newcommand{\kdim}{\ensuremath{\mathcal{K}\text{-dim }}}
\newcommand{\ve}{\varepsilon}
\newcommand{\entropy}{\ensuremath{\mathscr{H}_{\ve}}}

\begin{document}
\title{Pluripolarity of Manifolds} 
\author{Oleg Eroshkin}
\address{Department of Mathematics and Statistics\\ University Of New
  Hampshire\\ Durham, New Hampshire 03824}
\email{oleg.eroshkin@unh.edu}

\maketitle{}

\section{Introduction}
\label{sec:introduction}

A set $E\subset \bbC^n$ is called \emph{pluripolar} if there exists a
non-constant plurisubharmonic function $\phi$ such that $\phi\equiv
-\infty$ on $E$. Pluripolar sets form a natural category of ``small''
sets in complex analysis. Pluripolar sets are polar, so they have
Lebesgue measure zero, but there are no simple criteria to determine
pluripolarity.  In this paper we discuss the conditions that ensure
pluripolarity for smooth manifolds. This problem has a long
history. S.~Pinchuk \cite{Pinchuk:1974}, using Bishop's ``gluing
disks'' method, proved that a generic manifold of class $C^3$ is
non-pluripolar.  Recall that the manifold $M\subset \bbC^n$ is called
\emph{generic at a point} $p\in M$, if the tangent space $T_p M$ is
not contained in a proper complex subspace of $\bbC^n$. A.~Sadullaev
\cite{Sadullaev:1976}, using the same method proved that a subset of
positive measure of a generic manifold of class $C^3$ is non-pluripolar.

In the opposite direction E.~Bedford \cite{Bedford:1982} showed that a
real-analytic nowhere generic manifold is pluripolar. For some
applications to harmonic analysis the condition of real-analyticity is
too restrictive.  However, the result does not hold for merely smooth
manifolds.  K.~Diederich and J.~E.~Fornaess
\cite{Diederich:Fornaess:1982} found an example of a non-pluripolar
smooth curve in $\bbC^2$.  They construct a function $f\in
C^{\infty}[0,1]$ such that the graph of this function is not
pluripolar. In this example the derivatives $f^{(k)}$ grow very fast
as $k\to\infty$.

Recently, D.~Coman, N.~Levenberg and E.~A.~Poletsky \cite{CLP:2005}
proved that curves of Gevrey class $G^s$, $s<n+1$ in $\bbC^n$ are
pluripolar.  We generalize this result to higher dimensional
manifolds.  Recall that the submanifold $M\subset \bbC^n$ is called
\emph{totally real} if for every $p\in M$ the tangent space $T_p M$
contains no complex line.

\begin{theorem}
  \label{thm:GCP}
  Let $M\subset \bbC^n$ be a totally real submanifold of Gevrey class
  $G^s$. If $\dim M=m$ and $ms<n$, then $M$ is pluripolar.
\end{theorem}

In fact we prove a stronger result.  It follows from Theorem 2.1 in
\cite{AT:1984} that a compact set $X\subset \bbC^n$ is pluripolar if
and only if for any bounded domain $D$ containing $X$ and $\epsilon>0$
there exists polynomial $P$ such that
\[\sup\left\{|P(z)|:z\in D\right\}\geq 1\]
and
\[\sup\left\{|P(z)|:z\in X\right\}\leq \epsilon^{\deg P}\;.\]

\begin{theorem}
  \label{thm:SmallPolynomial}
  Let $M\subset\bbC^n$ be a totally real submanifold of Gevrey class
  $G^s$. Let $X$ be a compact subset of $M$. If $\dim M=m$ and $ms<n$,
  then for every $h<\frac{n}{ms}$ and every $N>N_0=N_0(h)$ there
  exists a non-constant polynomial $P\in\bbZ[z_1, z_2,\dots,z_n]$,
  $\deg P\leq N$ with coefficients bounded by $exp(N^h)$, such that

  \begin{equation}
    \label{eq:intsupexp}
    \sup\left\{|P(z)|:z\in X\right\}< \exp(-N^h)\;.
  \end{equation}
\end{theorem}

This result is similar to the construction of an auxiliary function in
transcendental number theory (cf.\cite{Waldschmidt:2000} Proposition
4.10).

The Theorem \ref{thm:GCP} gives some information about polynomially
convex hulls of manifolds of Gevrey class. Recall, that the
\emph{polynomially convex hull} $\widehat{X}$ of $X$ consists of all
$z\in\bbC^m$ such that
\[
|P(z)|\leq\sup_{\zeta\in X}|P(\zeta)|\;,
\]
for all polynomials $P$. It is well known, that the polynomially
convex hull of a pluripolar compact set is pluripolar (this follows
immediately from Theorem 4.3.4 in \cite{Hormander:1990}).

We also introduce the notion of \emph{Kolmogorov dimension} for a
compact subset $X\subset \bbC^n$ (denoted $\kdim X$) with the
following properties.
\begin{enumerate}
\item $0\leq \kdim X\leq n$.
\item $\kdim \widehat{X}=\kdim X$.
\item \[\kdim \bigcup_{j=1}^m X_j = \max\{\kdim X_j:j=1,\dots, m\}\;.\]
\item If $D\subset \bbC^n$ is a domain, $X\subset D$ and $\phi:D\to
  \bbC^k$ is a holomorphic map, then $\kdim \phi(X)\leq \kdim X$.
\item If $\kdim X<n$, then $X$ is pluripolar.
\end{enumerate}

The main result of this paper is the following estimate of the
Kolmogorov dimension of totally real submanifolds of Gevrey class.
\begin{theorem}
  \label{thm:main}
  Let $M\subset\bbC^n$ be a totally real submanifold of Gevrey class
  $G^s$. Let $X$ be a compact subset of $M$. If $\dim M=m$ then $\kdim
  X \leq ms$.
\end{theorem}

\begin{remark}
  This estimate is sharp. The similar estimates hold for more general
  class of $CR$-manifolds. These issues will be addressed in the
  forthcoming paper.
\end{remark}

In the next section we recall the definition and basic properties of
functions of Gevrey class.  The notion of Kolmogorov dimension of $X$
is defined in terms of $\ve$-entropy of traces on $X$ of bounded
holomorphic functions. The definition and basic properties of
$\ve$-entropy are given in Section \ref{sec:entropy}. In Section
\ref{sec:kdim} we discuss the notion of Kolmogorov dimension. The
proof of Theorem \ref{thm:main} is sketched in Section \ref{sec:main}.

The author wishes to thank the referee for useful comments and
numerous suggestions.

\section{Gevrey Class}
\label{sec:gevrey}

We need to introduce some notation first.
For a multi-indices $\alpha=(\alpha_1, \alpha_2,\dots,\alpha_m)$,
$\beta=(\beta_1, \beta_2, \dots,\beta_m)$
we define $|\alpha|=\sum_{j=1}^m \alpha_j$, 
$\alpha!=\prod_{j=1}^m \alpha_j!$, and
\[\binom{\alpha}{\beta}=\frac{\alpha!}{(\alpha-\beta)!\beta!}\;.\]
For an integer $k$ we define 
$\alpha+k=(\alpha_1+k, \alpha_2+k,\dots, \alpha_m+k)$.
For a point $x\in \bbR^m$ we define $x^{\alpha}= \prod_{j=1}^{m} x_j^{\alpha_j}$. If
$f\in C^{\infty}(\bbR^m)$ we denote 
\[
\partial^{\alpha}f=
\frac{\partial^{|\alpha|}}{\partial^{\alpha_1}x_1\dots\partial^{\alpha_m}x_m}f\;.
\]

Let $U$ be an open set in $\bbR^m$ and $s\ge 1$. A function 
$f\in C^{\infty}(U)$ is said to belong to Gevrey class $G^s(U)$ if for
every compact $K\subset U$ there exists a constant $C_K>0$ such that
\begin{equation}
  \label{eq:1}
  \sup_{x\in K} |\partial^{\alpha}f(x)|\leq C_K^{|\alpha|+1}(\alpha!)^s\;, 
\end{equation}
for every multi-index $\alpha$. The class $G^s$ forms an
algebra. The Gevrey class $G^s$ is closed
with respect to composition and the Implicit Function Theorem holds for
$G^s$ \cite{Komatsu:1979}, thus one may define manifolds (and
submanifolds) of Gevrey class $G^s$ in the usual way.

\section{The notion of $\varepsilon$-entropy}
\label{sec:entropy}

Let $(E,\rho)$ be a totally bounded metric space.  A family of sets
$\{C_j\}$ of diameter not greater than $2\ve$ is called an
\emph{$\ve$-covering} of $E$ if $E\subseteq \bigcup C_j$.  Let
$N_{\ve}(E)$ be the smallest cardinality of the $\ve$-covering.

A set $Y\subseteq E$ is called \emph{$\ve$-distinguishable} if the
distance between any two points in $Y$ is greater than $\ve$:
$\rho(x,y)>\ve$ for all $x,y\in Y$, $x\neq y$. Let $M_{\ve}(E)$ be the
largest cardinality of an $\ve$-distinguishable set.

For a nonempty totally bounded set $E$ the natural logarithm
\[
\entropy(E) = \log N_{\ve}(E)
\]
is called the \emph{$\ve$-entropy}.

The notion of $\ve$-entropy was introduced by A.~N.~Kolmogorov in the
1950's.  Kolmogorov was motivated by Vitushkin's work on Hilbert's
13th problem and Shannon's information theory. Note that Kolmogorov's
original definition (see \cite{Kolmogorov:Tihomirov:1959}) is slightly
different from ours (he used the logarithm to base 2). Here we follow
\cite{Lorentz:1996}.

We will need some basic properties of the $\ve$-entropy.
\begin{lemma}\label{lem:MeNeIneq}
  (see \cite{Kolmogorov:Tihomirov:1959}, Theorem \textup{IV}) For each
  totally bounded space $E$ and each $\ve>0$
  \begin{equation}
    \label{eq:MeNeIneq}
    M_{2\ve}(E)\leq N_{\ve}(E)\leq M_{\ve}(E)
  \end{equation}
\end{lemma}

% \begin{lemma}
%   \label{lem:ndim}
%   Let $V$ be an $n$-dimensional Banach space and $\mu$ a translation invariant
%   measure on $V$. If $K$ is a compact subset of $V$, then
%   \begin{equation}
%     \label{eq:ndim}
%     M_{\ve}(K)\geq\frac{\mu(K)}{\mu(B_{\ve})}\;,
%   \end{equation}
%   where $B_{\ve}$ is a ball of radius $\ve$.
% \end{lemma}
% \begin{proof}
%   If $U=\{x_j\in K:j=1\dots M\}$ is a maximal $\ve$-distinguishable
%   set with $M=M_{\ve}(K)$ elements, then balls $B_j$ of radius $\ve$
%   with centers at $x_j$ cover compact $K$. Hence
%   \[
%   \mu(K)\leq\sum_{j=1}^M \mu(B_j)=M\mu(B_{\ve})\;,
%   \]
%   and \[M_{\ve}(K)=M\geq\frac{\mu(K)}{\mu(B_{\ve})}\;.\]
% \end{proof}

\begin{lemma}
  \label{lem:Product}
  Let $\left\{(E_j,\rho_j):j=1,2,\dots,k\right\}$ be a family of
  totally bounded metric spaces. Let $(E,\rho)$ be a Cartesian product
  with a sup-metric, i.e.
  \[E=E_1\times E_2 \times \dots \times E_k\;,\]
  \[\rho((x_1, x_2,\dots ,x_k),(y_1, y_2,\dots ,y_k))=max_j
  \rho_j(x_j,y_j)\;.\] Then \[\entropy(E)\leq\sum_j \entropy(E_j)\].
\end{lemma}
\begin{proof}
  Let $\{C_{jl}\}$ $l=1,\dots N_j$ be an $\ve$-covering of $E_j$. Then
  the family \[\{C_{1l_1}\times C_{2l_2}\times\dots\times C_{kl_k}:
   l_j=1,\dots,N_j\}\] is an $\ve$-covering of $E$.
\end{proof}

We also need upper bounds for $\ve$-entropy of a ball in
finite-dimensional $\ell^\infty$ space. Let $\bbR^n_{\infty}$ be
$\bbR^n$ with the sup-norm:
\[||(x_1, x_2, \dots, x_n)||_{\infty}=\max_j |x_j|\;.\]
\begin{lemma}
  \label{lem:InftyBox}
  Let $B_r$ be a ball of radius $r$ in $\bbR^n_{\infty}$.  Then
  \[\entropy(B_r)<n\log\left(\frac{r}{\ve}+1\right)\;.\]
\end{lemma}
\begin{proof}
  The inequality is obvious for $n=1$. The general case then follows
  from Lemma \eqref{lem:Product}.
\end{proof}

\section{Kolmogorov Dimension}
\label{sec:kdim}
Let $X$ be a compact subset of a domain $D\subset \bbC^n$. Let $A_X^D$
be a set of traces on $X$ of functions analytic in $D$ and bounded by
$1$. So $f\in A_X^D$ if and only if there exists a function $F$
holomorphic on $D$ such that \[\sup_{z\in D} |F(z)|\leq 1\] and
$f(z)=F(z)$ for every $z\in X$. By Montel's theorem $A_X^D$ is a
compact subset of $C(X)$.

The connections between the asymptotics of $\ve$-entropy and the
pluripotential theory were predicted by Kolmogorov, who conjectured
that in the one dimensional case
\[
\lim_{\ve\to
  0}\frac{\entropy(A_X^D)}{\log^{2}(1/\ve)}=\frac{C(X,D)}{(2\pi)}\;,
\]
where $C(X,D)$ is the condenser capacity. This conjecture was proved
simultaneously by K.~I.~Babenko \cite{Babenko:1958} and V.~D.~Erokhin
\cite{Erokhin:1958} for simply-connected domain $D$ and connected
compact $X$ (see also \cite{Erokhin:1968}). For more general pairs
$(X,D)$ the conjecture was proved by Widom \cite{Widom:1972}
(simplified proof can be found in \cite{FM:1980}).

In the multidimensional case Kolmogorov asked to prove the existence
of the limit
\[\lim_{\ve\to
  0}\frac{\entropy(A_X^D)}{\log^{n+1}(1/\ve)}
\]
and to calculate it explicitly. V.~P.~Zahariuta \cite{Zahariuta:1984}
showed how the solution of Kolmogorov problem will follow from the
existence of the uniform approximation of the relative extremal
plurisubharmonic function $u^*_{X,D}$ by multipole pluricomplex Green
functions with logarithmic poles in $X$ (Zahariuta conjecture).  Later
this conjecture was proved by Nivoche \cite{Nivoche:2004} for a
``nice'' pairs $(D,X)$. Therefore it is established that for such
pairs
\[
\lim_{\ve\to
  0}\frac{\entropy(A_X^D)}{\log^{n+1}(1/\ve)}=\frac{C(X,D)}{(2\pi)^n}\;,
\]
where $C(X,D)$ is the relative capacity (see \cite{BT:1982}).

The pluripolarity of $X$ is equivalent to the condition $C(X,D)=0$
(\cite{BT:1982}).  If $\entropy(A_X^D)=o(\log^{n+1}(\frac{1}{\ve}))$
then $X$ is ``small'' (pluripolar) and the asymptotics of an
$\ve$-entropy can be used to determine how ``small'' $X$ is.  We will
use the function
\begin{equation}
  \label{eq:PSI}
  \Psi(X,D)=\limsup_{\ve\to 0}
  \frac{\log \entropy(A_X^D)}{\log\log\frac{1}{\ve}}-1 
\end{equation}
to characterize the ``dimension'' of $X$.

For a compact subset $X\subset\bbC^n$ we define the \emph{Kolmogorov
  dimension} $\kdim X=\Psi(X,D)$, where $D$ is a bounded domain
containing $X$. A.~N.~Kolmogorov proposed in
\cite{Kolmogorov:Tihomirov:1959} to use $\Psi(X,D)$ as a
\emph{functional dimension} of space of holomorphic functions on $X$.
The idea to use $\Psi(X,D)$ to characterize the ``size'' of compact
$X$ seems to be new.  We proceed to prove that $\Psi(X,D)$ is
independent of the bounded domain containing $X$.

\begin{lemma}
  \label{lem:union}
  Let $D\subset \bbC^n$ be a bounded domain. If $X_1,X_2,\dots, X_k$
  are compact subsets of $D$, then
  \[
  \Psi(\bigcup X_j, D) = \max \Psi(X_j, D)\;.
  \]

\end{lemma}
\begin{proof}
  Let $X=\bigcup X_j$. Embeddings $X_j\to X$ generate a natural
  isometric embedding $C(X)\to C(X_1)\times C(X_2)\times\dots\times
  C(X_k)$.  Restriction of this embedding on $A_X^D$ gives an
  isometric embedding $A_X^D\to A_{X_1}^D\times A_{X_2}^D\times
  \dots\times A_{X_k}^D$.  By Lemma \ref{lem:Product}
  \begin{equation}
    \label{eq:union}
    \entropy(A_X^D)\leq \sum_{j=1}^k\entropy(A_{X_j}^D)
    \leq k\max \entropy(A_{X_j}^D)\;,
  \end{equation}
  and the result follows.
\end{proof}

For a point $a\in \bbC^n$ and $R>0$, we denote $\Delta(a,R)$ the open
polydisk of radius $R$ with center at $a$
\[
\Delta(a,R)= \left\{z=(z_1, z_2, \dots, z_n): |z_j-a_j|<R, \;\;
  j=1,\dots,n\right\}\;.
\]
The following well known result follows directly from Cauchy's
formula.
\begin{lemma}
  \label{lem:approx}
  Let $R$ and $r$ be real numbers, $R>r>0$. Let $a\in\bbC^n$ and $f$
  be a bounded analytic function on the polydisk $\Delta(a,R)$. Then
  for any positive integer $k$ there exists a polynomial $P_k$ of
  degree $k$ such that
  \begin{equation}
    \label{eq:2}
    \sup_{w\in\Delta(a,r)}|f(w)-P_k(w)|
    \leq \frac{1}{R-r}\left(\frac{r}{R}\right)^k\sup_{z\in\Delta(a,R)}|f(z)|\;.
  \end{equation}
\end{lemma}

\begin{lemma}
  \label{lem:bernstein}
  Let $R$ and $r$ be real numbers, $R>r>0$. If a polynomial $P$ of
  degree $k$ satisfies the inequality $|P(w)|\leq A$ for every $w\in
  \Delta(a,r)$, then for $z\in \Delta(a,R)$ we have
  \begin{equation}
    \label{eq:3}
    |P(z)|\leq A\left(\frac{R}{r}\right)^k\;.
  \end{equation}
\end{lemma}
\begin{proof}
  Let $Q(\lambda)=P((\lambda(z-a)+a)$. The inequality follows from the
  application of maximum modulus principle to $Q(\lambda)/\lambda^k$.
\end{proof}

\begin{theorem}
  \label{thm:indep}
  Let $X$ be a compact in $\bbC^n$. If $D_1$, $D_2 \subset \bbC^n$ are
  bounded domains containing $X$, then $\Psi(X, D_1)=\Psi(X, D_2)$.
\end{theorem}

\begin{proof}
  Without loss of generality we may assume that $D_1\subset D_2$. Then
  $A_X^{D_1}\supset A_X^{D_2}$ and $\Psi(X, D_1)\geq\Psi(X, D_2)$. We
  establish the special case of two polydisks first. Suppose
  $D_1=\Delta(a,r)$ and $D_2=\Delta(a,R)$, where $R>r>0$. Choose
  $r'>0$ such that $r>r'$ and $X\subset\Delta(a,r')$. Let $\{f_1,
  f_2,\dots,f_N\}\subset A_X^{D_1}$ be a maximal $\ve$-distinguishable
  set with $N=M_{\ve}(A_X^{D_1})$. By Lemma \ref{lem:approx} there
  exist a positive integer $k$ and polynomials $\{p_1,
  p_2,\dots,p_N\}$ of degree $k$ such that
  \[
  \sup_{z\in X} |f_j(z)-p_j(z)|<\ve/3\qquad \text{for $j=1,2,\dots,N$}
  \]
  and $k\leq L\log \frac{1}{\ve}$, where $L$ depends only on $R$ and
  $r'$.  Then polynomials $\{p_1, p_2,\dots,p_N\}$ are
  $\ve/3$-distinguishable.  By Lemma \ref{lem:bernstein}, polynomials
  \[q_j=\left(\frac{r}{R}\right)^kp_j\] are bounded on $D_2$ by $1$.
  There exist positive constants $c,\lambda$, which depend only on
  $R,r$ and $r'$, but not on $\ve$, such that polynomials $\{q_1,
  q_2,\dots,q_N\}$ are $\delta$-distinguishable (as points in
  $A_X^{D_2}$), where $\delta=c \ve^{\lambda}$. Hence
  \begin{equation}
    \label{eq:4}
    N_{\delta}(A_X^{D_2})\leq M_{\delta}(A_X^{D_2})\leq N_{\ve}(A_X^{D_1})\;,
  \end{equation}
  and $\Psi(X, D_1)=\Psi(X, D_2)$.

  If $D_1=\Delta(a,r)$ and $D_2$ is an arbitrary bounded domain
  containing $D_1$, then there exists $R>0$, such that
  $D_3=\Delta(a,R)\supset D_2$. In this case the theorem follows from
  the inequalities
  \[\Psi(X, D_1)\geq\Psi(X, D_2)\geq\Psi(X, D_3)=\Psi(X, D_1)\;.\]

  Now consider the general case. Let polydisks $\Delta_1,
  \Delta_2,\dots, \Delta_s\subset D_1$ form an open cover of $X$.
  There exist compact sets $X_1, X_2, \dots X_s$ such that $X_j\subset
  \Delta_j$ for $j=1, 2,\dots, s$ and $\bigcup X_j = X$. Then
  \[
  \Psi(X_j,D_1)=\Psi(X_j,\Delta_j)=\Psi(X_j,D_2)\;, \qquad\text{for
    $j=1, 2,\dots, s$.}
  \]
  The theorem follows now from Lemma \ref{lem:union}.
\end{proof}

\begin{example}
  \label{ex:poly}
  Let $X=\overline{\Delta(0,r)}$. Let $R>r$ and $D=\Delta(0,R)$.
  Kolmogorov \cite{Kolmogorov:Tihomirov:1959} (see also
  \cite{Lorentz:1996}) showed that
  \begin{equation}
    \label{eq:8}
    \entropy(A_X^D) = C(n,r,R)\left(\log\frac{1}{\ve}\right)^{n+1}+
    \mathrm{O}\left(\left(\log\frac{1}{\ve}\right)^n
      \log\log\frac{1}{\ve}\right)\;.
  \end{equation}

  Therefore $\Psi(X,D)=n$ and $\kdim X = n$.
\end{example}

\begin{theorem}
  \label{thm:properties}
  Let $X$ be a compact subset of $\bbC^n$. The Kolmogorov dimension
  $\kdim X$ satisfies the following properties.
  \begin{enumerate}
  \item $0\leq \kdim X\leq n$.
  \item $\kdim \{z\}=0$.
  \item $\kdim \widehat{X}=\kdim X$.
  \item If $Y\subset X$ then $\kdim Y \leq \kdim X$.
  \item If $\{X_j\}$ is a finite family of compact subsets of
    $\bbC^n$, then
    \[\kdim \bigcup_{j=1}^m X_j = \max\{\kdim X_j:j=1,\dots, m\}\;.\]
  \item If $D\subset \bbC^n$ is a domain, $X\subset D$ and $\phi:D\to
    \bbC^k$ is a holomorphic map, then $\kdim \phi(X)\leq \kdim X$.
  \item If $\kdim X<n$, then $X$ is pluripolar.
  \end{enumerate}
\end{theorem}

\begin{remark}
  Property (5) does not hold for countable unions.  There exists a
  countable compact set $X$ such that $n=\kdim X$ (see Example
  \ref{ex:countable}). Such set $X$ also provides a counterexample to
  the converse of (7).
\end{remark}

\begin{proof}
  Properties (4) and (6) follow immediately from the definition.
  Property (2) follows from Lemma \ref{lem:InftyBox}.  Lemma
  \ref{lem:union} implies (5).  The inequality $\kdim X \geq 0$
  immediately follows from the definition. From \eqref{eq:8} follows
  that Kolmogorov dimension of a closed polydisk equal $n$. Therefore
  (4) implies the second part of (1).

  To show (3) consider a Runge domain $D$ containing $X$. Let
  $W=\widehat{X}$. Then $A_X^D$ and $A_W^D$ are isometric, hence
  $\Psi(X,D)=\Psi(\widehat{X},D)$ and (3) follows.

  The property (7) follows from the following theorem and Lemma
  \ref{lem:pluripolar}.
\end{proof}

\begin{theorem}
  \label{thm:subexp}
  Let $X$ be a compact subset of $\bbC^n$, such that $\kdim X=s<n$.
  Then for every $1<h<\frac{n}{s}$ and every $N>N_0=N_0(h)$ there
  exists a non-constant polynomial $P\in\bbZ[z_1, z_2,\dots,z_n]$,
  $\deg P\leq N$ with coefficients bounded by $\exp(N^h)$, such that

  \begin{equation}
    \label{eq:intsupexp2}
    \sup\left\{|P(z)|:z\in X\right\}< \exp(-N^h)\;.
  \end{equation}
\end{theorem}
\begin{proof}
  The result follows from Dirichlet principle. Let $D=\Delta(0,R)$ be
  a polydisk containing $X$. Assume that $R>1$. Let $\ve
  =\frac{1}{2}exp(-2N^h-N\log R-n\log N)$. Choose $t$ such that
  $s<t<\frac{n}{h}$. For large enough $N$ there exists an
  $\ve$-covering of $A_X^D$ with cardinality $\leq
  \exp\left\{(\log\frac{1}{\ve})^{t+1}\right\}$.

  Let $T=[exp(n^h)]$ (integer part of $exp(n^h)$). There are
  \[M=T^{\binom{N+n}{n}}\] polynomials of degree at most $N$ with
  coefficients in $\{1, 2, \dots T\}$.  Let $\{p_1, p_2,\dots p_M\}$
  be a list of all such polynomials. Clearly the polynomial
  \[q_j=\frac{1}{N^n R^N \exp(N^h)}p_j\] belongs to $A_X^D$. By our
  choice of $t$, $h(t+1)<n+h$, therefore there are more polynomials
  $q_j$ than cardinality of the $\ve$-covering, so there are two
  polynomials, let say $q_1$ and $q_2$ such that
  \[
  |q_1(z)-q_2(z)|\leq 2\ve\qquad\text{for every $z\in X$.}
  \]
  Then $P=p_1-p_2$ satisfies (\ref{eq:intsupexp2}).
\end{proof}

Let $\mathcal{P}_N$ be the set of all polynomials (with complex
coefficients) on $\bbC^n$ of the degree $\leq N$, whose supremum on a
unit polydisk $\Delta(0,1)$ is at least $1$.

\begin{corollary}
  \label{col:subexp}
  If $\kdim X=s<n$, then for every $1<h<\frac{n}{s}$ and every
  $N>N_0=N_0(h)$ there exists polynomial $P\in\mathcal{P}_N$, such
  that
  \begin{equation}
    \label{eq:col:supexp}
    \sup\left\{|P(z)|:z\in X\right\}< \exp(-N^h)\;.
  \end{equation}
\end{corollary}
Corollary \ref{col:subexp} may be used to bound Kolmogorov dimension
from below.
\begin{example}
  \label{ex:countable}
  Given $0<r<1$ and a positive integer $N$ there exists a finite set
  $X_{r,N}\subset \Delta(0,r)$, such that for any polynomial $P\in
  \mathcal{P}_N$, the following inequality holds
  \begin{equation}
    \label{eq:11}
    \max_{z\in X_{r,N}}|P(z)|\geq\frac{1}{2}r^N\;.
  \end{equation}
  For example, if $\ve=\frac{1-r}{2N}r^N$, then a maximal
  $\ve$-distinguishable subset of $\Delta(0,r)$ satisfies condition
  \eqref{eq:11}.  Let \[X=\bigcup_{k=2}^{\infty} X_{1/k,k}\;.\] Then
  $X$ is compact and for any $N>2$ and $P\in \mathcal{P}_N$
  \[\sup_{z\in X}|P(z)|\geq \max_{z\in X_{1/N,N}}|P(z)|\geq
  \frac{1}{2}\left(\frac{1}{N}\right)^N\;.\] Therefore by Corollary
  \ref{col:subexp} $\kdim X =n$.
\end{example}

To finish the proof of Theorem \ref{thm:properties} we need the
following well-known result.
\begin{lemma}
  \label{lem:pluripolar}
  Let $X$ be a compact subset of $\bbC^n$. If there exists a sequence
  $\{a_k\}$, $a_k>0$ and a family of polynomials $P_k\in\mathcal{P}_k$
  such that
  \begin{equation}
    \label{eq:5}
    \sup_{z\in X} |P_k(z)|\leq e^{-a_k}\;, \qquad\text{and}
  \end{equation}
  \begin{equation}
    \label{eq:6}
    \lim \frac{a_k}{k} = \infty\;,
  \end{equation}
  then $X$ is pluripolar.
\end{lemma}
\begin{proof}
  Let $v_k(z)=\frac{1}{a_k}\log P_k(z)$ and $v(z)=\limsup v_k(z)$. We
  will show that $v\geq-2/3$ on a dense set.  Let $\zeta\in \bbC^n$
  and $0<\delta<1$. Suppose that $\Delta(\zeta,R)\supset \Delta(0,1)$.
  We will show that there exists a nested sequence of closed polydisks
  $\Delta_m=\overline{\Delta(w_m,\delta_m)}$ with
  $\Delta_1=\overline{\Delta(\zeta,\delta)}$, and an increasing
  sequence of positive integers $k_1=1<k_2<\dots<k_m<\dots$ such that
  $v_{k_m}\geq -2/3$ on $\Delta_m$ for $m>1$. Given
  $\Delta=\Delta_m=\overline{\Delta(w_m,\delta_m)}$ by Lemma
  \ref{lem:bernstein} for any given $k$ there exists $w\in\Delta$ such
  that
  \begin{equation}
    \label{eq:walsh:bernstein}
    |P_k(w)|\geq \left(\frac{\delta_m}{R+\delta}\right)^k\;.
  \end{equation}
  Choose $k=k_{m+1}>k_m$ such that
  \[\frac{a_k}{k}\geq 2\log\frac{R+\delta}{\delta_m}\;.\]
  Then by \eqref{eq:walsh:bernstein} $v_k(w)\geq -1/2$. Choose
  $w_{m+1}=w$. Because the function $v_k$ is continuous at $w$, there
  exists a closed polydisk
  $\Delta_{m+1}=\overline{\Delta(w_{m+1},\delta_{m+1})}\subset
  \Delta_m$, such that $v_k\geq -2/3$ on $\Delta_{m+1}$.  Therefore
  $v\geq-2/3$ on a dense set. By \eqref{eq:5}, $v|_X\leq -1$ and so
  $X$ is a negligible set. By \cite{BT:1982}, negligible sets are
  pluripolar and result follows.
\end{proof}

\begin{remark}
  This lemma and the converse follow from Theorem 2.1 in
  \cite{AT:1984}.
\end{remark}

\section{Manifolds of Gevrey Class}
\label{sec:main}
In view of Theorem \ref{thm:subexp} and Theorem \ref{thm:properties}
(6), Theorems \ref{thm:SmallPolynomial} and \ref{thm:GCP} are
corollaries of Theorem \ref{thm:main}.  In this section we prove
Theorem \ref{thm:main}.

Let $M\subset\bbC^n$ be an $m$-dimensional totally real submanifold of
Gevrey class $G^s$. Let $X\subset M$ be a compact subset. Fix $p\in
M$. There exist holomorphic coordinates $(z,w)=(x+iy,w)\in\bbC^n$,
$x$,$y\in \bbR^m$, $w\in\bbC^{n-m}$ near $p$, vanishing at $p$,
real-valued functions of class $G^s$ $h_1$, $h_2$,\dots, $h_m$, and
complex valued functions of class $G^s$ $H_1$, $H_2$,\dots,$H_{n-m}$
such that $h_1'(0)=h_2'(0)=\dots=h_m'(0)=0$,
$H_1'(0)=H_2'(0)=\dots=H_{n-m}'(0)=0$, and locally
\begin{equation}
  \label{eq:coord}
M=\left\{(x+iy,w): y_j=h_j(x), w_k=H_k(x)\right\}\;.  
\end{equation}
For smooth manifold the existence of such coordinates is well known
(see, for example \cite{BER:1999}, Proposition 1.3.8). Note, that 
functions $h_j$ and $H_k$ are defined by Implicit Function Theorem,
and so by \cite{Komatsu:1979} are of class $G^s$.

We fix such coordinates and choose $r$ sufficiently small.
In view of Theorem \ref{thm:properties} (5), it is sufficient to prove
Theorem \ref{thm:main} for $X\subset\Delta(p,r)$. 
Put $D=\Delta(p,1)$. To estimate $\Psi(X,D)$ we will cover $X$ by
small balls, approximate functions in $A_X^D$ by Taylor polynomials,
and then replace in these polynomials terms $w^\lambda$ and $y^\nu$ by
Taylor polynomials of functions $H^\lambda$ and $h^\nu$.
To estimate the Taylor coefficients for powers of functions of Gevrey
class we need the following lemma.
\begin{lemma}
  \label{lem:power:gevrey}
  If $f\in G^s(K)$ and $|f|\leq 1$ on $K$, then there exist a constant $C$
  such that for any positive integer $k$ and any multi-index $\alpha$
  the following inequality holds on $K$ 
\begin{equation}
  \label{eq:power}
  |\partial^\alpha f^k|\leq C^{|\alpha|}\binom{\alpha+k-1}{\alpha} (\alpha!)^s\;.
\end{equation}
\end{lemma}

Recall, that $\alpha+k=(\alpha_1+k, \alpha_2+k,\dots,\alpha_m+k)$.
\begin{proof}
We will argue by induction on $k$. Because $|f|\leq 1$,
there exists a constant $C$, such that 
\[\partial^\alpha f\leq C^{|\alpha|}(\alpha!)^s\] and \eqref{eq:power}
holds for $k=1$. Suppose \eqref{eq:power} holds for $1, 2,\dots,k$, then
\begin{align*}
  |\partial^\alpha f^{k+1}|&=\left|\sum_{\nu\leq\alpha}\binom{\alpha}{\nu}
  \partial^\nu f^k \partial^{\alpha-\nu} f\right| \leq C^{|\alpha|}
  \sum_{\nu\leq\alpha} \binom{\alpha}{\nu}\binom{\nu+k-1}{\nu}
  (\nu!)^s \left((\alpha-\nu)!\right)^s \\
    &=C^{|\alpha|} \sum_{\nu\leq\alpha}
    \binom{\nu+k-1}{\nu}\binom{\alpha}{\nu}\nu!(\alpha-\nu)! 
    (\nu!)^{s-1} \left((\alpha-\nu)!\right)^{s-1}\\
    &\leq C^{|\alpha|} \sum_{\nu\leq\alpha}
    \binom{\nu+k-1}{\nu}\alpha! 
    (\alpha!)^{s-1}=C^{|\alpha|}\binom{\alpha+k}{\alpha} (\alpha!)^s
\end{align*}
\end{proof}

\begin{remark}
  The same proof holds for the product of $k$ different functions, provided
  that they satisfy the Gevrey class condition (\ref{eq:1}) with the
  same constant $C_K$.
\end{remark}
Let $t>s\geq 1$ and $N$ be a large integer, which will tend to infinity
later. Fix positive $a<t-s$. Put $\delta=N^{1-t}$ and $\ve= N^{-aN}$.
We may cover $X$ by less than $(1/\delta)^m$ balls of radius
$\delta$.  Let $Q$ be one of these balls and $K$ be the set of
restrictions on $Q$ of functions in $A_X^D$. We claim that any
function $f$ in $K$ may be approximated by polynomials in $x_1,
x_2,\dots, x_m$ of the degree $\leq N$ with coefficients bounded by
$C^N (N!)^{s-1}$ with error less than $2\ve$, where constant $C$
depends on $X$ and $r$ only. Let us show how the theorem follows from
this claim.  The real dimension of the space of polynomials of the
degree $\leq N$ is $T=2\binom{N+m}{N}$. Consider in the
$T$-dimensional space with the sup-norm
$\bbR^T_{\infty}$ the ball $B$ of a radius $C^N (N!)^{s-1}$. By Lemma
\ref{lem:InftyBox}, 
\[\entropy(B)\leq 2\binom{N+m}{N}\log\left(\frac{C^N
  (N!)^{s-1}}{\ve}+1\right)=O( N^{m+1}\log N)\;. 
\]
By the claim $\ve$-covering of $B$ generate $3\ve$-covering of $K$,
therefore \[\mathscr{H}_{3\ve}(K)=O( N^{m+1}\log N)\;.\]
Then by \eqref{eq:union} 
\begin{equation}
  \label{eq:7}
  \entropy(A_X^D)=O\left(\left(\frac{1}{\delta}\right)^m N^{m+1}\log
  N\right)=O(N^{mt+1}\log N)\;.
\end{equation}
Now we let $N$ tend to infinity. By (\ref{eq:7}), $\kdim
X=\Psi(X,D)\leq mt$. The only restriction imposed on $t$ so far was
$t>s$. Hence $\kdim X\leq ms$. It remains to prove the claim.
We approximate a function $f$ in $K$ in two steps.
Consider the Taylor polynomial $P$ of $f$ centered at the center of
the ball $Q$ of the degree $N$ .
% \[P(z)=\sum_{|\alpha|\leq N} \partial^{\alpha}
% f(z_j,w_j)\frac{(z-z_j)^{\alpha}}{\alpha!}\frac{ \;.\]
By Cauchy's formula 
\[\sup_{Q_l}|f-P|<\frac{1}{1-r-\delta}\left(\frac{\delta}{1-r}\right)^N<\ve\]
for sufficiently large $N$. Suppose 
$P(z,w)=\sum c_{\lambda\mu\nu}x^\lambda y^\mu w^\nu$. Because $f\in
A_X^D$, $|c_{\lambda\mu\nu}|\leq 1$.

On the next step we approximate $y^\mu$ and $w^\nu$ by the Taylor
polynomials of the degree $N$ of $h^\mu$ and $H^\nu$. Let $(x_0,y_0,w_0)$ be the
center of the ball $Q$. Let $g$ be one of the functions $h_1$,
$h_2$,\dots $h_m$, $H_1$, $H_2$,\dots,$H_{n-m}$ and $L\leq N$.
Then by Taylor formula
\[
g^L(x_0+h)=\sum_{|\alpha|\leq N} \partial^\alpha
f(x_0)\frac{h^\alpha}{\alpha!} +R_N(x,h)\;.
\]
By Lemma \ref{lem:power:gevrey} for $||h||_{\infty}<\delta$
\[|R_N(x,h)|\leq C^{N+1}\delta^N\sum_{|\alpha|=N+1} \binom{\alpha+N-1}{\alpha}
(\alpha!)^{s-1}\;.\]
Therefore $\log|R_N(x,h)|=(s-t+o(1))N\log N$ and claim follows.

\end{document}